\documentclass[a4paper,oneside,reqno]{amsart}

\usepackage[T1]{fontenc}
\usepackage[utf8]{inputenc}
\usepackage{lmodern}
\usepackage[top=30mm,bottom=30mm,left=30mm,right=30mm]{geometry}
\usepackage{mathtools,amssymb}
\usepackage[foot]{amsaddr}
\usepackage[shortlabels]{enumitem}
\usepackage{microtype}
\usepackage{xcolor}
\usepackage[
 backend=biber,
 style=alphabetic,
 sorting=nyvt
]{biblatex}
\usepackage[
 bookmarksnumbered=true,
 colorlinks=true,
 allcolors=blue,
 pdftitle={Extension of solvable coverings across horizontal divisors over the integers},
 pdfauthor={Naganori Yamaguchi}
]{hyperref}

\title[Extension of solvable coverings]{Extension of solvable coverings across horizontal divisors over the integers}
\date{}
\thanks{This work was supported by the Japan Society for the Promotion of Science (JSPS) KAKENHI Grant Number 23KJ0881.}

\author[N.~Yamaguchi]{Naganori~Yamaguchi}
\address{Tokyo Denki University, 5 Senju-Asahi-cho, Adachi-ku, Tokyo 120-8551, Japan}
\email{n.yamaguchi@mail.dendai.ac.jp}

\newcommand{\etale}{\mathrm{\acute{e}t}}
\newcommand{\abelian}{\mathrm{ab}}
\DeclareMathOperator{\Spec}{Spec}
\DeclareMathOperator{\Frac}{Frac}
\DeclareMathOperator{\Supp}{Supp}
\DeclareMathOperator{\GL}{GL}
\newcommand{\AbsGalGrp}[1]{G_{#1}}
\newcommand{\Abelianization}[1]{#1^{\abelian}}
\newcommand{\EtFundGrp}[1]{\pi_{1}^{\etale}(#1)}
\newcommand{\CardinalityOfSet}[1]{\lvert #1\rvert}

\numberwithin{equation}{section}

\theoremstyle{plain}
\newtheorem{theorem}{Theorem}[section]

\newtheorem{Itheorem}{Theorem}

\newtheorem{Icorollary}[Itheorem]{Corollary}

\newtheorem{proposition}[theorem]{Proposition}
\newtheorem{lemma}[theorem]{Lemma}
\newtheorem{corollary}[theorem]{Corollary}

\theoremstyle{definition}
\newtheorem{definition}[theorem]{Definition}
\newtheorem{example}[theorem]{Example}
\newtheorem{remark}[theorem]{Remark}

\begin{document}

\begin{abstract}
	We study the effect of deleting a horizontal divisor from an arithmetic scheme on its \'etale fundamental group.
	Assume that the scheme is regular along the divisor, that each prime component contains a normal crossings point of residue characteristic $2$, and that the generic fibers of the boundary components remain geometrically integral after finite \'etale base change.
	We prove that the kernel of the induced surjection on fundamental groups is topologically perfect.
	In particular, the inclusion induces an isomorphism on maximal prosolvable quotients.
	We apply this result to projective hyperplane complements and to the extension of representations over complete local coefficient rings.
\end{abstract}

\maketitle
\tableofcontents

\section*{Introduction}

We study finite \'etale coverings of complements of horizontal divisors over $\mathbb{Z}$.
A theorem of T.~Saito, recalled by Y.~Ihara in~\cite[Appendix]{MR1305400}, states that the complement of the sections $0,1,\infty$ in $\mathbb{P}^{1}_{\mathbb{Z}}$ has trivial \'etale fundamental group.
Its proof uses the absolute Abhyankar lemma at every prime to show that the inertia groups are trivial.
J.-B.~Bost and F.~Charles give a related result for arithmetic surfaces in~\cite[Proposition~9.3.6]{bost2022quasiprojectiveformalanalyticarithmeticsurfaces}.
The same local argument applies in arbitrary dimension when each boundary component contains a normal crossings point in every residue characteristic; see Proposition~\ref{prop:boundary-invariance}.

For finite solvable coverings, a condition at the prime $2$ is sufficient in some cases.
R.~Shimizu and the author proved in~\cite[Theorem~2.6]{shimizu2026branchedcoveringsprojectiveline} that, for a reduced finite union $D$ of rational sections of $\mathbb{P}^{1}_{\mathbb{Z}}$,
\begin{equation*}
	\EtFundGrp{\mathbb{P}^{1}_{\mathbb{Z}}\setminus D}^{\mathrm{solv}}=1
	\quad\Longleftrightarrow\quad
	D\text{ is normal crossings at the prime }2.
\end{equation*}
Here, for a profinite group $G$ we write $G^{\mathrm{solv}}$ for its maximal prosolvable quotient.
In the present paper, we extend the reverse implication to horizontal divisors in arbitrary dimension.

\begin{Itheorem}[Theorem~\ref{thm:perfect-kernel}]\label{thm:intro-perfect-kernel}
	Let $X$ be a connected locally Noetherian scheme over $\Spec(\mathbb{Z})$, let $D$ be a horizontal effective Cartier divisor on $X$, and set $U=X\setminus\Supp(D)$.
	Assume the following conditions:
	\begin{enumerate}[label=\textup{(\alph*)}]
		\item\label{cond:intro-regular}
		      $X$ is regular along $D$.
		\item\label{cond:intro-nc-two}
		      Every prime component of $D$ contains a point of residue characteristic $2$ at which $D$ has normal crossings.
		\item\label{cond:intro-geometric}
		      For every connected finite \'etale covering $X'\to X$, every prime component of $D\times_X X'$ has geometrically integral generic fiber over $\mathbb{Q}$.
	\end{enumerate}
	Then $U$ is connected, and the inclusion $U\hookrightarrow X$ induces an isomorphism
	\begin{equation*}
		\EtFundGrp{U}^{\mathrm{solv}}
		\xrightarrow{\ \sim\ }\EtFundGrp{X}^{\mathrm{solv}}.
	\end{equation*}
\end{Itheorem}

Condition~\ref{cond:intro-geometric} is satisfied if every prime component of $D$ is smooth, of finite type, and surjective over $\Spec(\mathbb{Z})$, with geometrically connected generic fiber; see Remark~\ref{rem:smooth-boundary}.

The proof first considers a finite abelian quotient.
At each boundary component, the inertia order is odd by the normal crossings condition~\ref{cond:intro-nc-two} in characteristic $2$, and it divides $2$ by condition~\ref{cond:intro-geometric}.
Hence boundary inertia has trivial image.
Applying this argument after finite \'etale base change shows that the image of the boundary kernel in every finite quotient is perfect.
Hence the kernel is topologically perfect.
Therefore, the asserted morphism is an isomorphism.

For a projective hyperplane arrangement in $\mathbb{P}^n_{\mathbb{Z}}$ with $n\geq2$, if no codimension-two linear subspace of the geometric generic fiber is contained in three hyperplanes, then the \'etale fundamental group of the geometric generic complement is abelian by the theorem of L\^{e}--Saito \cite[Theorem~\textup{(3.1.2)}]{MR735874}.
Hence the main theorem gives the following corollary.

\begin{Icorollary}[{Corollary~\ref{cor:hyperplane-complements}}]
	Let $\mathcal{A}$ be a projective hyperplane arrangement over $\mathbb{Z}$.
	Set
	\begin{equation*}
		D_{\mathcal{A}}=\sum_{H\in\mathcal{A}}H,\qquad
		U_{\mathcal{A}}=\mathbb{P}^n_{\mathbb{Z}}\setminus\Supp(D_{\mathcal{A}}).
	\end{equation*}
	Assume that every member of $\mathcal{A}$ contains a point of residue characteristic $2$ at which $D_{\mathcal{A}}$ has normal crossings.
	Then
	\begin{equation*}
		\EtFundGrp{U_{\mathcal{A}}}^{\mathrm{solv}}=1.
	\end{equation*}
	If, moreover, $n\geq 2$ and the geometric generic arrangement is generic in codimension two, then
	\begin{equation*}
		\EtFundGrp{U_{\mathcal{A}}}=1.
	\end{equation*}
\end{Icorollary}

The second corollary concerns a continuous representation $\rho:\EtFundGrp{U}\to\GL_n(A)$, where $A$ is a complete Noetherian local ring with finite residue field.
Under the assumptions of Theorem~\ref{thm:intro-perfect-kernel}, $\rho$ extends to $X$ if and only if its residual representation does; see Corollary~\ref{cor:residual-extension}.

\medskip

The paper is organized as follows.
Section~\ref{sec:local} explains basic applications of the absolute Abhyankar lemma.
Section~\ref{sec:dyadic} proves Theorem~\ref{thm:intro-perfect-kernel} and its two corollaries.

\section{Local ramification and Abhyankar's lemma}\label{sec:local}

Unless otherwise stated, all schemes are over $\Spec(\mathbb{Z})$.
We give prime components of divisors their reduced structure.
For an integral scheme $E$, we write $K(E)$ for its function field.

Let $X$ be a connected locally Noetherian scheme.
By~\cite[Expos\'e~V, \S~7]{MR0354651}, the category of finite \'etale coverings of $X$ is a Galois category.
Hence, for a geometric point $\bar\eta \to X$, we denote the associated profinite group by
\begin{equation*}
	\pi_{1}^{\etale}(X,\bar\eta)
\end{equation*}
and call it the \emph{\'etale fundamental group} of $X$.
Since this group is independent of the base point up to inner automorphism, we fix a geometric point $\bar\eta \to X$ and often omit $\bar\eta$ from the notation.

Let $U\subset X$ be a connected dense open subscheme.
Let $\xi\in X\setminus U$ be a point such that $\mathcal{O}_{X,\xi}$ is a discrete valuation ring.
Choose a strict henselization $\mathcal{O}_{X,\xi}^{\mathrm{sh}}$ and set $K_{\xi}^{\mathrm{sh}}\coloneq\Frac(\mathcal{O}_{X,\xi}^{\mathrm{sh}})$.
We write
\begin{equation*}
	I_{\xi}\coloneq\mathrm{Im}\left(\EtFundGrp{\Spec(K_{\xi}^{\mathrm{sh}})}
	\to
	\EtFundGrp{U}\right)
\end{equation*}
and call it the \emph{inertia subgroup attached to $\xi$}.
This subgroup is well defined up to conjugacy.
For a finite quotient $q\colon\EtFundGrp{U}\twoheadrightarrow G$ and the corresponding connected Galois covering $W\to U$, we write
\begin{equation*}
	I_{\xi}(G)\coloneq I_{\xi}(W/U)\coloneq q(I_{\xi}).
\end{equation*}
If $E$ is a prime divisor with generic point $\xi_{E}$ and $X$ is regular at $\xi_{E}$, then we write $I_{E}$ instead of $I_{\xi_{E}}$.

\begin{lemma}\label{lem:boundary-quotient}
	Let $X$ be a connected locally Noetherian scheme, and let $D$ be an effective Cartier divisor on $X$.
	Set $U\coloneq X\setminus\Supp(D)$.
	Assume that $X$ is regular along $D$, that is, $\mathcal{O}_{X,x}$ is regular for every $x\in\Supp(D)$.
	Then $U$ is connected, and the inclusion $U\hookrightarrow X$ induces a surjection
	\begin{equation*}
		\EtFundGrp{U}\twoheadrightarrow\EtFundGrp{X}.
	\end{equation*}
	Moreover, its kernel is the closed normal subgroup generated by the inertia subgroups $I_E$, where $E$ ranges over the prime components of $D$.
\end{lemma}

\begin{proof}
	Since $D$ is an effective Cartier divisor, the open immersion $U\hookrightarrow X$ is quasi-compact.
	For each $x\in\Supp(D)$, let $f\in \mathcal{O}_{X,x}^{\mathrm{sh}}$ be a local equation of the pullback of $D$.
	Then
	\begin{equation*}
		U\times_X\Spec(\mathcal{O}_{X,x}^{\mathrm{sh}})=\Spec(\mathcal{O}_{X,x}^{\mathrm{sh}}[1/f]).
	\end{equation*}
	Since $\mathcal{O}_{X,x}^{\mathrm{sh}}$ is a regular local domain and $f\neq 0$, this scheme is nonempty and connected.
	Hence restriction of finite \'etale coverings from $X$ to $U$ is fully faithful by~\cite[Expos\'e~XIV, Corollaire~1.4 and Remarques~1.6\textup{(b)}]{MR476737}.
	Full faithfulness gives a bijection
	\begin{equation*}
		\operatorname{Hom}_X(X,X\amalg X)
		\xrightarrow{\ \sim\ }
		\operatorname{Hom}_U(U,U\amalg U).
	\end{equation*}
	These morphisms correspond to open-and-closed decompositions of $X$ and $U$, respectively.
	Since $X$ is connected, $U$ has no nontrivial open-and-closed decomposition.
	Thus $U$ is connected.
	Hence the homomorphism $\EtFundGrp{U}\to\EtFundGrp{X}$ is surjective by~\cite[Expos\'e~V, Proposition~6.9]{MR0354651}.

	By the Zariski--Nagata purity theorem \cite[Expos\'e~X, Th\'eor\`eme~3.1]{MR0354651}, applied over the regular local schemes $\Spec(\mathcal{O}_{X,x})$ for $x\in\Supp(D)$, a connected finite \'etale Galois covering $W\to U$ extends locally across $D$ if and only if $I_E(W/U)=1$ for every prime component $E$ of $D$.
	These local extensions glue uniquely by full faithfulness.
	Hence the description of the kernel follows.
\end{proof}

For an effective Cartier divisor $D$ on $X$, we say that $D$ has \emph{strict normal crossings at} $x\in D$ if $\mathcal{O}_{X,x}$ is regular and a local equation of $D$ is a product of distinct members of a regular system of parameters.
We say that $D$ has \emph{normal crossings at} $x$ if its pullback to an \'etale neighbourhood of $x$ is a strict normal crossings divisor.
The divisor $D$ is called a \emph{normal crossings divisor} if it has this property at every point of its support (see~\cite[Definition~1.8.2]{MR316453}).

A prime divisor is called \emph{horizontal} if it dominates $\Spec(\mathbb{Z})$.
An effective Cartier divisor is called \emph{horizontal} if all its prime components are horizontal.

\begin{lemma}\label{lem:inertia-control}
	Let $X$ be a connected locally Noetherian scheme, and let $D$ be a horizontal effective Cartier divisor on $X$.
	Set $U\coloneq X\setminus\Supp(D)$.
	Let $E$ be a prime component of $D$.
	Let $p$ be a prime.
	Assume the following conditions:
	\begin{enumerate}[label=\textup{(\alph*)}]
		\item\label{cond:inertia-regular}
		      $X$ is regular along $D$.
		\item\label{cond:inertia-nc}
		      $E$ contains a point of residue characteristic $p$ at which $D$ has normal crossings.
	\end{enumerate}
	Then, for every connected finite \'etale Galois covering $W\to U$, we have
	\begin{equation*}
		p\nmid \CardinalityOfSet{I_E(W/U)}.
	\end{equation*}
\end{lemma}

\begin{proof}
	The scheme $U$ is connected by Lemma~\ref{lem:boundary-quotient}.
	Choose a point $x\in E$ as in condition~\ref{cond:inertia-nc}.
	By the normal crossings assumption, $\mathcal{O}_{X,x}^{\mathrm{sh}}$ is a regular local ring, and the pullback of $D$ to $\Spec(\mathcal{O}_{X,x}^{\mathrm{sh}})$ is defined by $t_1\cdots t_s=0$, where $t_1,\ldots,t_s$ form part of a regular system of parameters.
	Hence we obtain
	\begin{equation*}
		V\coloneq U\times_X\Spec(\mathcal{O}_{X,x}^{\mathrm{sh}})
		=\Spec\bigl(\mathcal{O}_{X,x}^{\mathrm{sh}}[1/(t_1\cdots t_s)]\bigr).
	\end{equation*}
	The pullback $W\times_U V\to V$ is finite \'etale.
	For each $i$, the generic point of $V(t_i)$ lies over the generic point of a horizontal component of $D$, and hence has residue characteristic zero.
	Therefore this covering is tamely ramified along each $V(t_i)$.
	The absolute Abhyankar lemma \cite[Expos\'e~XIII, Appendice~I, Proposition~5.2]{MR0354651} shows that all its ramification indices along the divisors $V(t_i)$ are prime to $p$.
	Choose a component $E'$ of the pullback of $D$ to $\Spec(\mathcal{O}_{X,x}^{\mathrm{sh}})$ lying over $E$.
	Strict localization is an inverse limit of \'etale neighbourhoods.
	Ramification indices are unchanged under \'etale base change and on passing to this limit.
	Hence the ramification index along $E'$ is equal to that along $E$.
	The latter equals $\CardinalityOfSet{I_E(W/U)}$, since $E$ has generic residue characteristic zero.
	This proves the assertion.
\end{proof}

The following proposition generalizes~\cite[Proposition~9.3.6]{bost2022quasiprojectiveformalanalyticarithmeticsurfaces}.

\begin{proposition}\label{prop:boundary-invariance}
	Let $X$ be a connected locally Noetherian scheme, and let $D$ be a horizontal effective Cartier divisor on $X$.
	Set $U\coloneq X\setminus\Supp(D)$.
	Assume the following conditions:
	\begin{enumerate}[label=\textup{(\alph*)}]
		\item\label{cond:invariance-regular}
		      $X$ is regular along $D$.
		\item\label{cond:invariance-nc}
		      For every prime $p$, every prime component of $D$ contains a point of residue characteristic $p$ at which $D$ has normal crossings.
	\end{enumerate}
	Then $U$ is connected, and the inclusion $U\hookrightarrow X$ induces an isomorphism
	\begin{equation*}
		\EtFundGrp{U}\xrightarrow{\ \sim\ }\EtFundGrp{X}.
	\end{equation*}
\end{proposition}

\begin{proof}
	By Lemma~\ref{lem:boundary-quotient}, the scheme $U$ is connected, the natural homomorphism $\EtFundGrp{U}\to\EtFundGrp{X}$ is surjective, and the kernel is the closed normal subgroup generated by the inertia subgroups.
	Let $W\to U$ be a connected finite \'etale Galois covering.
	Assume that $I_{E}(W/U)$ is nontrivial for some prime component $E$ of $D$, and let $p$ be a prime divisor of its order.
	By condition~\ref{cond:invariance-nc}, there exists a point of $E$ of residue characteristic $p$ at which $D$ has normal crossings.
	By Lemma~\ref{lem:inertia-control}, the order of $I_{E}(W/U)$ is prime to $p$, a contradiction.
	Hence all boundary inertia groups have trivial image in every finite quotient of $\EtFundGrp{U}$, and therefore the kernel of the natural homomorphism $\EtFundGrp{U}\to\EtFundGrp{X}$ is also trivial.
	This proves the assertion.
\end{proof}

\begin{example}\label{ex:section1}
	The following are standard applications of Proposition~\ref{prop:boundary-invariance}.
	\begin{enumerate}[(1)]
		\item\label{ex:section1_1}
		      We recall that
		      \begin{equation*}
			      \EtFundGrp{\Spec(\mathbb{Z})}
			      =\EtFundGrp{\mathbb{P}^{m}_{\mathbb{Z}}}=1.
		      \end{equation*}
		      The equality $\EtFundGrp{\Spec(\mathbb{Z})}=1$ follows from Minkowski's theorem \cite[Chapter~III, Theorem~2.18]{MR1697859}.
		      The assertion for projective space follows from the proper homotopy exact sequence and its geometric simple connectedness \cite[Expos\'e~X, Corollaire~1.4, and Expos\'e~XI, Proposition~1.1]{MR0354651}.
		      Applying Proposition~\ref{prop:boundary-invariance} to a hyperplane at infinity also gives $\EtFundGrp{\mathbb{A}^{m}_{\mathbb{Z}}}=1$.
		\item
		      Let $n\geq3$, and let $M_{0,n,\mathbb{Z}}$ be the moduli scheme of smooth proper curves of genus zero with $n$ ordered disjoint marked points.
		      Then we have
		      \begin{equation*}
			      \EtFundGrp{M_{0,n,\mathbb{Z}}}=1.
		      \end{equation*}
		      Indeed, the stable compactification $\overline{M}_{0,n,\mathbb{Z}}$ is smooth and proper over $\mathbb{Z}$ with geometrically connected fibers, and its boundary is a relative normal crossings divisor whose components are surjective over $\Spec(\mathbb{Z})$ (see~\cite[Theorem~2.7]{MR0702953} and~\cite[Facts~1 and~2]{MR1034665}).
		      By Proposition~\ref{prop:boundary-invariance}, the inclusion $M_{0,n,\mathbb{Z}}\hookrightarrow\overline M_{0,n,\mathbb{Z}}$ induces an isomorphism on fundamental groups.
		      Sending the first three markings to $0,1,\infty$ identifies $M_{0,n,\mathbb{Q}}$ with an open subscheme of $\mathbb{A}^{n-3}_{\mathbb{Q}}$.
		      Hence the geometric generic fiber of the compactification is simply connected by birational invariance (see~\cite[Expos\'e~X, Corollaire~3.4]{MR0354651}).
		      The assertion follows from the proper homotopy exact sequence~\cite[Expos\'e~X, Corollaire~1.4]{MR0354651} and the equality in~\ref{ex:section1_1}.
		\item
		      Let $a,b\geq2$ be coprime integers.
		      Then
		      \begin{equation*}
			      \EtFundGrp{\Spec\mathbb{Z}[x,y,1/(x^a-y^b)]}=1.
		      \end{equation*}
		      Indeed, $D=V(x^a-y^b)$ is a horizontal prime divisor in $\mathbb{A}^{2}_{\mathbb{Z}}$.
		      For every prime $p$, consider the point with maximal ideal $\mathfrak{m}=(p,x-1,y-1)$.
		      In $\mathfrak{m}/\mathfrak{m}^2$, the class of the equation of $D$ is
		      \begin{equation*}
			      a(x-1)-b(y-1).
		      \end{equation*}
		      This class is nonzero since $p$ cannot divide both $a$ and $b$.
		      Hence the equation is part of a regular system of parameters, and therefore $D$ has strict normal crossings at this point.
		      The assertion follows from Proposition~\ref{prop:boundary-invariance} and the equality in~\ref{ex:section1_1}.
	\end{enumerate}
\end{example}

\section{Prosolvable boundary invariance at residue characteristic \texorpdfstring{$2$}{2}}\label{sec:dyadic}

In this section, we prove the main theorem of the present paper.
We begin with the abelian case.
The following lemma gives sufficient conditions for deleting a horizontal divisor to leave the abelianization of the \'etale fundamental group unchanged.

\begin{lemma}\label{lem:abelian-invariance}
	Let $X$ be a connected locally Noetherian scheme, and let $D$ be a horizontal effective Cartier divisor.
	Set $U\coloneq X\setminus\Supp(D)$.
	Assume the following conditions:
	\begin{enumerate}[label=\textup{(\alph*)}]
		\item\label{cond:abelian-regular}
		      $X$ is regular along $D$.
		\item\label{cond:abelian-nc-two}
		      Every prime component of $D$ contains a point of residue characteristic $2$ at which $D$ has normal crossings.
		\item\label{cond:abelian-geometric}
		      Every prime component of $D$ has geometrically integral generic fiber.
	\end{enumerate}
	Then $U$ is connected, and the inclusion $U\hookrightarrow X$ induces an isomorphism
	\begin{equation*}
		\Abelianization{\EtFundGrp{U}}
		\xrightarrow{\ \sim\ }
		\Abelianization{\EtFundGrp{X}}.
	\end{equation*}
\end{lemma}

\begin{proof}
	By Lemma~\ref{lem:boundary-quotient}, the scheme $U$ is connected, the natural homomorphism $\EtFundGrp{U}\to\EtFundGrp{X}$ is surjective, and the kernel is the closed normal subgroup generated by the inertia subgroups.
	Let $G$ be a finite abelian quotient of $\EtFundGrp{U}$, let $E$ be a prime component of $D$, and put $r=\CardinalityOfSet{I_E(G)}$.
	By condition~\ref{cond:abelian-nc-two}, there is a point of $E$ of residue characteristic $2$ at which $D$ has normal crossings.
	Lemma~\ref{lem:inertia-control}, applied at this point, gives $2\nmid r$.

	Put $k=K(E)$.
	Since $E$ is horizontal, $\operatorname{char}(k)=0$.
	The local description of tame inertia gives an isomorphism
	\begin{equation*}
		I_E(G)\simeq\mu_r(k^{\mathrm{sep}})
	\end{equation*}
	under which the action induced by local Galois conjugation on the left corresponds to the natural action of $\AbsGalGrp{k}$ on the right (see~\cite[Expos\'e~X, \S~3]{MR0354651}).
	Since $G$ is abelian, conjugation acts trivially on $I_E(G)$.
	Hence every $r$-th root of unity is fixed by $\AbsGalGrp{k}$ and therefore belongs to $k$.
	By condition~\ref{cond:abelian-geometric}, $\mathbb{Q}$ is algebraically closed in $k$.
	Therefore the only roots of unity in $k$ are $1$ and $-1$.
	Since $k$ contains a primitive $r$-th root of unity, we obtain $r\mid2$.
	Together with $2\nmid r$, this gives $r=1$.
	Therefore every boundary inertia subgroup has trivial image in $G$.
	The description of the kernel above then shows that $\EtFundGrp{U}\to\EtFundGrp{X}$ induces an isomorphism on abelianizations.
\end{proof}

We turn to maximal prosolvable quotients.
For this purpose, we need the abelian result to hold after every connected finite \'etale covering of $X$.
Conditions~\ref{cond:abelian-regular} and~\ref{cond:abelian-nc-two} are preserved under such coverings.
However, the geometric integrality condition~\ref{cond:abelian-geometric} does not hold for the prime components of the pulled-back divisor in general.
We therefore require this condition after every connected finite \'etale base change, as in the following definition.

\begin{definition}\label{def:boundary-condition}
	Let $X$ be a connected locally Noetherian scheme, and let $D$ be a horizontal effective Cartier divisor.
	We say that $(X,D)$ satisfies condition~$(\dagger)$ if the following conditions hold:
	\begin{enumerate}[label=\textup{(\alph*)}]
		\item\label{cond:boundary-regular}
		      $X$ is regular along $D$.
		\item\label{cond:boundary-nc-two}
		      Every prime component of $D$ contains a point of residue characteristic $2$ at which $D$ has normal crossings.
		\item\label{cond:boundary-geometric}
		      For every connected finite \'etale covering $X'\to X$, every prime component of $D\times_X X'$ has geometrically integral generic fiber.
	\end{enumerate}
\end{definition}

\begin{remark}\label{rem:smooth-boundary}
	Let $E$ be a scheme smooth and of finite type over $\Spec(\mathbb{Z})$.
	Assume that $E\to\Spec(\mathbb{Z})$ is surjective and that $E_{\mathbb{Q}}$ is geometrically connected.
	Then every irreducible component of every finite \'etale covering of $E$ has geometrically integral generic fiber.
	Since $E$ is regular and connected, it suffices to consider a connected finite \'etale covering $T\to E$.
	By the homotopy exact sequence of Katz--Lang \cite[Preliminaries, Lemma~2(2)]{MR659153} and Example~\ref{ex:section1}\ref{ex:section1_1}, the natural homomorphism
	\begin{equation*}
		\EtFundGrp{E_{\overline{\mathbb{Q}}}}
		\to \EtFundGrp{E}
	\end{equation*}
	is surjective.
	Hence $T_{\overline{\mathbb{Q}}}$ is connected.
	Since it is also smooth over $\overline{\mathbb{Q}}$, it is integral.
	Thus a sufficient condition for condition~\ref{cond:boundary-geometric} is that every prime component $E$ of $D$ be smooth, of finite type, and surjective over $\Spec(\mathbb{Z})$, with geometrically connected generic fiber.
\end{remark}

\begin{lemma}\label{lem:finite-etale-stability}
	Let $(X,D)$ satisfy condition~$(\dagger)$, and let $X'\to X$ be a connected finite \'etale covering.
	Then $(X',D\times_{X} X')$ satisfies condition~$(\dagger)$.
\end{lemma}

\begin{proof}
	We set $D'\coloneq D\times_{X} X'$.
	Since flat base change preserves effective Cartier divisors, $D'$ is an effective Cartier divisor.
	Moreover, \'etale base change preserves regularity.
	In particular, $X'$ is regular along $D'$.
	This proves condition~\ref{cond:boundary-regular}.
	Condition~\ref{cond:boundary-geometric} is preserved by definition.

	Finally, we prove that condition~\ref{cond:boundary-nc-two} holds for each prime component of $D'$.
	Let $E'$ be such a component, and let $\eta'$ be its generic point.
	The image of $\eta'$ under the finite flat morphism $D'\to D$ is the generic point $\eta$ of a prime component $E$ of $D$.
	Since $D'\to D$ is finite, the image of $E'$ is closed.
	Thus $E'\to E$ is surjective.
	In particular, $E'$ is horizontal, and there exists a point $x'\in E'$ above the point $x\in E$ in condition~\ref{cond:boundary-nc-two}.
	Then $x'$ has residue characteristic $2$, and $D'$ has normal crossings at $x'$ by \'etale base change.
	This proves condition~\ref{cond:boundary-nc-two} for $(X',D')$.
\end{proof}

A profinite group $G$ is called \emph{topologically perfect} if $G=\overline{[G,G]}$.
The following is the main theorem of the present paper.

\begin{theorem}\label{thm:perfect-kernel}
	Let $X$ be a connected locally Noetherian scheme, and let $D$ be a horizontal effective Cartier divisor.
	We set $U\coloneq X\setminus\Supp(D)$.
	Assume that $(X,D)$ satisfies condition~$(\dagger)$.
	Then $U$ is connected, and the kernel of the natural homomorphism
	\begin{equation*}
		\EtFundGrp{U}
		\to
		\EtFundGrp{X}
	\end{equation*}
	induced by the inclusion $U\hookrightarrow X$ is topologically perfect.
	In particular, the homomorphism induces an isomorphism
	\begin{equation*}
		\EtFundGrp{U}^{\mathrm{solv}}
		\xrightarrow{\ \sim\ }
		\EtFundGrp{X}^{\mathrm{solv}}.
	\end{equation*}
\end{theorem}

\begin{proof}
	By Lemma~\ref{lem:boundary-quotient}, the scheme $U$ is connected, and the natural homomorphism $f\colon\EtFundGrp{U}\to\EtFundGrp{X}$ is surjective.
	Let $H\lhd\EtFundGrp{X}$ be an open normal subgroup, and let $X_H\to X$ be the corresponding connected finite \'etale covering.
	Put $D_H=D\times_X X_H$ and $U_H=U\times_X X_H$.
	By Lemma~\ref{lem:finite-etale-stability}, $(X_H,D_H)$ satisfies condition~$(\dagger)$, and the scheme $U_H$ is connected by Lemma~\ref{lem:boundary-quotient}.
	With $\EtFundGrp{U_H}$ and $\EtFundGrp{X_H}$ identified with $f^{-1}(H)$ and $H$, respectively, Lemma~\ref{lem:abelian-invariance} therefore gives an isomorphism
	\begin{equation*}
		\Abelianization{f^{-1}(H)}\xrightarrow{\ \sim\ }\Abelianization{H}.
	\end{equation*}
	Hence $\ker(f)\subset\overline{[f^{-1}(H),f^{-1}(H)]}$ for every such $H$.

	Let $q:\EtFundGrp{U}\twoheadrightarrow Q$ be any finite quotient, and put $N=q(\ker(f))$.
	The quotient $Q/N$ is a quotient of $\EtFundGrp{X}$.
	Taking $H$ to be the kernel of this quotient, we have $q(f^{-1}(H))=N$, and hence
	\begin{equation*}
		N=q(\ker(f))\subset
		q\bigl(\overline{[f^{-1}(H),f^{-1}(H)]}\bigr)=[N,N].
	\end{equation*}
	Thus $N$ is perfect.
	Since this holds in every finite quotient of $\EtFundGrp{U}$, it follows that $\ker(f)=\overline{[\ker(f),\ker(f)]}$, which proves the perfectness assertion.
	If $Q$ is solvable, then $N$ is both perfect and solvable, and hence $N=1$.
	Therefore every finite solvable quotient of $\EtFundGrp{U}$ factors through $\EtFundGrp{X}$, which proves the last assertion.
\end{proof}

Finally, we give two corollaries of Theorem~\ref{thm:perfect-kernel}.
For the first corollary, we apply the topological perfectness of the kernel in Theorem~\ref{thm:perfect-kernel} to representations over complete local rings.

\begin{corollary}\label{cor:residual-extension}
	We keep the notation and assumptions of Theorem~\ref{thm:perfect-kernel}.
	Let $(A,\mathfrak{m},k)$ be a complete Noetherian local ring with finite residue field of characteristic $p$.
	For $n\geq1$, let
	\begin{equation*}
		\rho:\EtFundGrp{U}\to\GL_n(A)
	\end{equation*}
	be a continuous representation, and let $\bar\rho:\EtFundGrp{U}\to\GL_n(k)$ be its reduction.
	Then $\rho$ factors through $\EtFundGrp{X}$ if and only if $\bar\rho$ does.
	When the factorization exists, it is unique.
\end{corollary}

\begin{proof}
	It suffices to prove the reverse implication.
	Put $K=\ker(\EtFundGrp{U}\to\EtFundGrp{X})$.
	By Theorem~\ref{thm:perfect-kernel}, $K$ is topologically perfect.
	If $\bar\rho$ factors through $\EtFundGrp{X}$, then
	\begin{equation*}
		\rho(K)\subset\ker\bigl(\GL_n(A)\to\GL_n(k)\bigr).
	\end{equation*}
	This congruence kernel is pro-$p$.
	Since a topologically perfect profinite group has no nontrivial pro-$p$ quotient, $\rho(K)=1$.
	Therefore $\rho$ factors through $\EtFundGrp{X}$.
	The factorization is unique since $\EtFundGrp{U}\to\EtFundGrp{X}$ is surjective.
\end{proof}

Next, we apply Theorem~\ref{thm:perfect-kernel} to complements of hyperplane arrangements.
Let $n\geq1$.
A \emph{projective hyperplane arrangement over $\mathbb{Z}$} is a finite set $\mathcal{A}=\{H_i\}_{i\in I}$ of distinct relative hyperplanes in $\mathbb{P}^n_{\mathbb{Z}}$; that is, its members are defined by equations
\begin{equation*}
	H_i=V(\ell_i),\qquad
	\ell_i=a_{i0}X_0+\cdots+a_{in}X_n,
\end{equation*}
where $a_{ij}\in\mathbb{Z}$, $\gcd(a_{i0},\ldots,a_{in})=1$ for every $i$, and $\ell_i\neq\pm\ell_j$ for $i\neq j$.
We say that the geometric generic arrangement is \emph{generic in codimension two} if no codimension-two linear subspace of $\mathbb{P}^n_{\overline{\mathbb{Q}}}$ is contained in three of its hyperplanes.

\begin{corollary}\label{cor:hyperplane-complements}
	Let $\mathcal{A}$ be a projective hyperplane arrangement over $\mathbb{Z}$.
	Set
	\begin{equation*}
		D_{\mathcal{A}}=\sum_{H\in\mathcal{A}}H,\qquad
		U_{\mathcal{A}}=\mathbb{P}^n_{\mathbb{Z}}\setminus\Supp(D_{\mathcal{A}}).
	\end{equation*}
	Assume that every member of $\mathcal{A}$ contains a point of residue characteristic $2$ at which $D_{\mathcal{A}}$ has normal crossings.
	Then
	\begin{equation*}
		\EtFundGrp{U_{\mathcal{A}}}^{\mathrm{solv}}=1.
	\end{equation*}
	If, moreover, $n\geq 2$ and the geometric generic arrangement is generic in codimension two, then
	\begin{equation*}
		\EtFundGrp{U_{\mathcal{A}}}=1.
	\end{equation*}
\end{corollary}

\begin{proof}
	Since each $H_i$ is isomorphic to $\mathbb{P}^{n-1}_{\mathbb{Z}}$, condition~\ref{cond:boundary-geometric} holds for $(\mathbb{P}^n_{\mathbb{Z}},D_{\mathcal{A}})$ by Remark~\ref{rem:smooth-boundary}.
	The remaining conditions in~$(\dagger)$ follow from the hypotheses.
	The first assertion therefore follows from Theorem~\ref{thm:perfect-kernel} and Example~\ref{ex:section1}\ref{ex:section1_1}.

	For the second assertion, we may assume that $\mathcal{A}\neq\varnothing$.
	The morphism $U_{\mathcal{A}}\to\Spec(\mathbb{Z})$ is smooth, of finite type, and surjective, with geometrically connected generic fiber.
	The homotopy exact sequence \cite[Preliminaries, Lemma~2(2)]{MR659153} and Example~\ref{ex:section1}\ref{ex:section1_1} give a surjection
	\begin{equation*}
		\EtFundGrp{U_{\mathcal{A},\overline{\mathbb{Q}}}}
		\twoheadrightarrow\EtFundGrp{U_{\mathcal{A}}}.
	\end{equation*}
	Since the geometric generic arrangement is generic in codimension two, L\^{e}--Saito's theorem \cite[Theorem~\textup{(3.1.2)}]{MR735874} shows that $\EtFundGrp{U_{\mathcal{A},\overline{\mathbb{Q}}}}$ is abelian.
	Thus $\EtFundGrp{U_{\mathcal{A}}}$ is abelian and is therefore trivial by the first assertion.
\end{proof}

\section*{Acknowledgements}
The author would like to thank Professor Akio Tamagawa for his valuable comments.

\printbibliography

\end{document}